\documentclass[11pt]{amsart}
\usepackage{amsmath,amssymb,amsthm,geometry,hyperref,enumitem,mathtools}
\title[The Dual of Quantifier Elimination: Boolean Elimination over
$\mathbb{C}$ and $\mathbb{R}$]{The Dual of Quantifier Elimination:\\Boolean
Elimination over $\mathbb{C}$ and $\mathbb{R}$}
\author{Matthew Frank}
\date{December 22, 2025}
\address{mfrank@uchicago.edu}

\newtheorem{theorem}{Theorem}[section]

\theoremstyle{remark}

\newtheorem{corollary}{Corollary}[section]
\newtheorem{lemma}{Lemma}[section]

\newcommand{\ab}{\mathbf{a}}
\newcommand{\AEf}{\texorpdfstring{\forall\exists}{AE}}
\newcommand{\EA}{\texorpdfstring{\exists\forall}{EA}}
\newcommand{\Ed}{\texorpdfstring{\exists^d}{E\string^d}}
\newcommand{\E}{\texorpdfstring{\exists}{E}}
\newcommand{\C}{\texorpdfstring{\mathbb{C}}{C}}
\newcommand{\R}{\texorpdfstring{\mathbb{R}}{R}}
\newcommand{\Q}{\texorpdfstring{\mathbb{Q}}{Q}}

\newcommand{\Ac}{\mathcal{A}}
\newcommand{\Lc}{\mathcal{L}}

\newcommand{\M}{\mathcal{M}}

\begin{document}

\begin{abstract}
We show that every finite Boolean combination of polynomial equalities and
inequalities in $\C^n$ admits two uniform normal forms: an $\exists\forall$
form and a $\forall\exists$ form, each using a \emph{single} polynomial
equation. Both forms use only one existentially quantified variable and one
universally quantified variable. Optimality results demonstrate that
no purely existential or universal normal form is possible over $\C$.

These results extend to sets constructible from entire functions,
and to quantifier-free formulas in functional languages over infinite fields of
characteristic $0$. A corollary shows Zilber's conjecture on quasiminimality
equivalent to its subcase quantifying a single equation.

Over $\R$ and $\Q$, similar results hold, including a singly-quantified $\exists$
form for Boolean combinations of equations and inequations, an $\exists$ form
for $\R$ established by prior methods, and other results for order inequalities
parallel to the forms over $\C$.

These results provide a dual to classical quantifier elimination: 
instead of removing quantifiers at the cost of increased Boolean complexity,
they remove Boolean structure at the cost of a short, fixed quantifier prefix.
The constructions have linear degree bounds and are explicit.
\end{abstract}

\maketitle

\section{Introduction}

Quantifier elimination replaces quantified formulas by quantifier-free
ones, typically at a cost of more complicated Boolean combinations.
This paper studies the complementary direction:
\emph{eliminating Boolean structure} at the cost of introducing
a short block of quantifiers.

Our main result here shows that over $\C$, every finite Boolean
combination of polynomial equations and inequations can be expressed
using one polynomial equation with one existential quantifier and one
universal quantifier, with the quantifiers in either order. More precisely:

\begin{theorem}\label{main}[Boolean Elimination over $\C$]
Let $\varphi(x)$ be any Boolean combination of polynomial equations and
inequations in variables $x\in\mathbb{C}^n$.
Then there are polynomials
$$p, q \in\mathbb{C}[x,a,b]$$
such that for all $x\in\mathbb{C}^n$,
\begin{align*}
\varphi(x)&\Longleftrightarrow(\exists a\in\C)(\forall b\in\C)\,p(x,a,b)=0\\
\varphi(x)&\Longleftrightarrow(\forall a\in\C)(\exists b\in\C)\,q(x,a,b)=0
\end{align*}
\end{theorem}

For example, if
$$\varphi(y,z) := (y=0 \wedge z\neq 0) \vee (z=0 \wedge y\neq 0)$$
then the constructions here show
$$\varphi \iff (\exists a)(\forall b)\Big[(1-az)+by\Big]
\Big[(1-ay)+bz\Big]=0$$
and also
$$\varphi \iff (\forall a)(\exists b)\Big[1-b(a-1)(a-2)\Big]
\Big[(a-2)yz+(a-1)(1-by)(1-bz)\Big]=0.$$

Even more complicated formulas admit normal forms with only $\exists$
and one $\forall$, in either order; having both normal forms provides flexibility
for use inside larger formulas. Furthermore the degrees of the
resulting polynomials are linearly bounded by the length of the original
formula in conjunctive or disjunctive normal form.

The main result may also be generalized complex-geometrically or
model-theoretically, with the same proofs.

\begin{theorem}\label{complex-geometric}[Reformulation in Complex Geometry]
Let $\Ac$ be an algebra of entire functions on $\C^n$, closed under polynomial
operations. Let $\Ac[a,b]$ denote the extension of this algebra by polynomials
in two complex variables $a$ and $b$. Let $S\subset\C^n$ be any constructible
set defined over $\Ac$, i.e. a finite Boolean combination of zero-sets
$\{x:f(x)=0\}$ where $f\in\Ac$. Then $S$ can be represented as a locus
of identical vanishing: there is a function $P\in\Ac[a,b]$ such that
$$x\in S\Longleftrightarrow \exists a\in\C \text{ where }
P(x,a,b) \text{ vanishes identically for all }b\in\C.$$
Dually, $S$ can be represented as a locus of universal solvability:
there is a function $Q\in\Ac[a,b]$ such that
$$x\in S\Longleftrightarrow \forall a\in\C, \text{ the function }
b\mapsto Q(x,a,b) \text{ has a zero in }\C.$$
\end{theorem}

The applies to several possible algebras, including the algebra of polynomials,
the algebra of exponential polynomials, the algebra of entire functions
of order $<\rho$ (for some $\rho>0$), the algebra of entire functions of
finite order, and the algebra of all entire functions.

\begin{theorem}\label{model-theoretic}[Reformulation in Model Theory]
Let $\Lc$ be a functional language containing $1,+,-,\times$.
Let $\M$ be an $\Lc$-structure satisfying
the axioms for an infinite field of characteristic $0$. Let $\varphi$ be any
quantifier-free formula in $\Lc$ over the variables $x_1,\ldots,x_n$.
Then there are $\Lc$-terms $p,q$ formed from polynomial
combinations of $1,a,b$ and the terms occurring in $\varphi$, such that
\begin{align*}
\M\models\varphi(x)&\Longleftrightarrow
(\exists a)(\forall b)\,p(x_1,\ldots,x_n,a,b)=0\\
\M\models\varphi(x)&\Longleftrightarrow
(\forall a)(\exists b)\,q(x_1,\ldots,x_n,a,b)=0
\end{align*}
\end{theorem}

This applies to several possible languages, including the pure polynomial
language, the language adding $\exp$, the language adding $1/\Gamma$,
the language adding all univariate entire functions, and the language of all
entire functions in any number of variables. In particular, we obtain an
equivalence for Zilber's conjecture on the quasiminimality of $\exp$ (on
which see Kirby's survey for more context).

\begin{corollary}
Let $\Lc$ be the language $\{+,-,\times,\exp\}$, and let quantifiers $Q_i$
be either $\forall$ or $\exists.$ Then the following are equivalent:
\begin{itemize}
\item
For any $n$, any sequence of quantifiers $Q_i$,
and any quantifier-free formula $\varphi(z,x_1,\ldots,x_n)$ in $\Lc$, the set
$$\{z \in C:Q_1x_1\ldots Q_n x_n\, \varphi(z,x_1,\ldots,x_n)\}\phantom{=0}$$
is either countable or cocountable.
\item
For any $n$, any sequence of quantifiers $Q_i$,
and any term $t(z,x_1,\ldots,x_n)$ in $\Lc$, the set
$$\{z \in C:Q_1x_1\ldots Q_n x_n\, t(z,x_1,\ldots,x_n)=0\}$$
is either countable or cocountable.
\end{itemize}
\end{corollary}

In other words, Zilber's conjecture is equivalent to its subcase on the
quasiminimality of sets defined by quantifying a single equation,
rather than sets defined by quantifying an arbitrary formula.

\pagebreak

The rest of the paper is organized as follows: Section 2 establishes notation.
Section 3 reviews context. Sections 4 and 5 construct the $\exists\forall$
and $\forall\exists$ polynomials for $\C$ explicitly, and provide bounds
on their degrees. Section 6 shows that these quantifier patterns are
optimal: there are simple formulas over $\C$ that cannot be expressed
by purely existential or purely universal quantification of a single equation.

Section 7 turns to other formally real fields, and shows that any Boolean
combination of equations and inequations ($u\neq 0$) over them can
be represented with a single $\exists$ and a single equality. Section 8
constructs the $\exists^d$ and $\forall\exists$ polynomials for Boolean
combinations over $\R$ that may involve inequalities ($u>0$).
Sections 9 and 10 explore prior results over $\R$ and forms over $\Q$.
Section 11 concludes with directions for further research.

\section{Notation}

We use the following notation:
\begin{align*}
a,b &\in \C \text{ (scalar quantified variables)},\\
d,e,f &\in \mathbb{N},\\
h & \in \{1,\ldots,d\}, \ h\neq i,\\
i,j,k &\in \{1,\ldots,d\}, \{1,\ldots,e\}, \{1,\ldots,f\} \text{ respectively},\\
m,n &\in \mathbb{N},\\
p,q &\in \text{polynomials},\\
r,s &\in \R,\\
t,u &\in \text{terms using }x,\\
v,w &\in \Q,\\
x &\in \C^n,\\
y,z &\in \C.
\end{align*}
The constructions below are uniform in the dimension of $\C^n$, and where
we need to refer to a vector there, we use $x$. The quantified variables
$a,b,r,s,v,w,y,z$ always range over scalars, whether in $\C$, $\R$, or $\Q$. 

\section{Context}

There are three key ingredients in the constructions and proofs here.

First is the algebraic encoding of propositional combinations
and inequalities over $\R$, which are used here for
$\R$ and as inspiration for $\C$:
\begin{lemma}
For all $t_1,t_2,u\in\R$:
\begin{align*}
t_1=0 \vee t_2=0 &\iff t_1t_2=0\\
t_1=0 \wedge t_2=0 &\iff t_1^2+t_2^2=0\\
u\neq0 &\iff (\exists b)\, bu=1\\
u\ge0 &\iff (\exists s)\, u=s^2\\
u>0 &\iff (\exists s)\, s^2u=1
\end{align*}
\end{lemma}
Since only the first and third of these hold over $\C$, an explicit method
for encoding propositional structure requires additional ingredients.

The second key ingredient here is Lagrange interpolation. An
ordinary example is that, if we want to construct a polynomial $p(a)$
with specified values $t_1,t_2,t_3$ at the arguments $1,2,3$, we can use
$$p(a)=
t_1 \frac{(a-2)(a-3)}{(1-2)(1-3)} +
t_2 \frac{(a-1)(a-3)}{(2-1)(2-3)} +
t_3 \frac{(a-1)(a-2)}{(3-1)(3-2)}.$$
The application here is simpler, without the divisions.
\begin{lemma}
The polynomial
$$p(a)=\sum_i t_i \prod_{h\neq i}(a-h)$$
vanishes at $a=i$ iff $t_i=0$.
\end{lemma}

The third key ingredient is basic quantifier elimination over an
algebraically closed field, where we need only the elimination of a
single existentially quantified equation:
\begin{lemma}
 For any polynomial $q(z)=\sum_{i=1}^d q_i z^i$,
\begin{align*}
(\exists z\in\C)\,(q(z)=0) &\iff
q_0=0 \vee q_1\neq 0 \vee \cdots \vee q_d\neq 0\\
(\forall z\in\C)\,(q(z)\neq0) &\iff
q_0\neq0 \wedge q_1= 0 \wedge \cdots \wedge q_d= 0
\end{align*}
\end{lemma}
This restates the fundamental theorem of algebra, and we use it for the
impossibility result on $\exists$. 
We do not use the longer and more complicated results of real quantifier
elimination here. (See Swan's expository article on quantifier elimination
over $\R$ and $\C$ for general background, and see Basu's survey for more
details on the complexity, which contrasts with the linear bounds here.)

We call the transformations of this paper \emph{Boolean elimination}
because they provide a dual to quantifier elimination: instead of
removing quantifiers at the cost of more propositional complexity,
we remove propositional complexity at the cost of a short, fixed
quantifier prefix. 

\section{Normalizing to $\EA$ over $\C$}

The constructions in this paper all assume the input formula is presented
in conjunctive or disjunctive normal form. The two quantified variables
in the resulting $\exists\forall$ and $\forall\exists$ forms correspond
to the two Boolean layers in these representations, either
disjunctions of conjunctions or conjunctions of disjunctions.

We start here with $\varphi$ in disjunctive normal form:
\begin{align*}
\varphi
&:=\bigvee_i\Big(\bigwedge_j t_{ij}= 0\;\wedge\;\bigwedge_k u_{ik}\neq 0\Big)
\end{align*}
The idea is to find a polynomial which encodes each disjunct by a factor
that vanishes identically as a polynomial in $b$ iff that disjunct holds.

This $\varphi$ is equivalent to the $\exists\forall$ form below:
\begin{align*}
\varphi\iff\exists a\,\forall b\,
\prod_i\Big[(1-a\prod_k u_{ik})+\sum_j t_{ij} b^j\Big]=0
\end{align*}

For the forward direction, suppose the disjunction holds, and suppose
specifically that the $i^{th}$ disjunct holds. We choose
$a=\prod_k u_{ik}^{-1}$. For that $a$ and for any $b$, the $i^{th}$ factor
in the above expression will vanish, so the $\exists\forall$ statement holds.

For the backwards direction, suppose that the $\exists\forall$ statement
holds for some $a$. Then for any $b$, one of the factors must vanish, and by
the pigeonhole principle, one of the factors must vanish for infinitely many
$b$. As a nonzero univariate polynomial has only finitely many roots,
if the $i^{th}$ factor vanishes for infinitely many $b$, then it must
be the zero polynomial in $b$. Therefore $1-a\prod_k u_{ik}=0$, which
implies that each $u_{ik}$ is nonzero; and also for each $j$, $t_{ij}=0$.
This establishes that the $i^{th}$ disjunct holds, and therefore that the
whole disjunction holds.

This proves:
\begin{theorem}\label{EA}
Let $\varphi(x)$ be a finite combination of polynomial equations and
inequations in disjunctive normal form, with $d$ disjuncts
and a total of $e$ equations.
Then there is a polynomial $p(a,b,x)\in\mathbb{C}[a,b,x]$,
of degree $d$ in $a$,
of degree $e$ in $b$,
and such that for all $x\in\mathbb{C}^n$,
$$\varphi(x)\Longleftrightarrow(\exists a\in\C)(\forall b\in\C)\,p(a,b,x)=0.$$
\end{theorem}

\section{Normalizing to $\AEf$ over $\C$}

We start with $\varphi$ in conjunctive normal form:
$$\varphi:=\bigwedge_i \Big(\bigvee_j t_{ij}=0\;\vee\;
\bigvee_k u_{ik}\neq 0\Big)$$

We use $\C$ having characteristic $0$ to assign a distinct integer parameter
to each conjunct, and apply Lagrange interpolation to construct a polynomial
whose $i^{th}$ term contributes only when $a$ equals $i$.
Then $\varphi$ is equivalent to the $\forall\exists$ form below:

\begin{align*}
\varphi \iff \forall a\,\exists b\,
\Big[1-b\prod_i(a-i)\Big]
\Big[\sum_i\prod_{h\neq i}(a-h)\prod_j t_{ij}\prod_k(1-bu_{ik})\Big]=0.
\end{align*}

For the forward direction, suppose the conjunction holds.
\begin{itemize}
\item If $a$ is distinct from all of the $i$'s, then let
$b=\prod_i(a-i)^{-1}$, so that the first bracketed expression will vanish.
\item If $a$ is equal to one of the $i$'s, and $t_{ij}=0$, then let $b=0$
(or any other value), and the second bracketed expression will vanish.
\item If $a$ is equal to one of the $i$'s, and $u_{ik}\neq 0$, then let
$b=u_{ik}^{-1}$, and the second bracketed expression will vanish.
\end{itemize}
In any case the $\forall\exists$ statement holds.

For the backwards direction, suppose that the $\forall \exists$ statement
holds, and consider $a=i$. Then the big product reduces to
$$\prod_{h\neq i}(i-h)\,\prod_j t_{ij}
\prod_k (1-b u_{ik})$$
and must vanish. So either $\prod_j t_{ij}$ vanishes,
in which case one of the $t_i$'s vanishes;
or $\prod_k (1-b u_{ik})$ vanishes,
in which case one of the $u_{ik}$ is non-zero.
In either case the conjunction holds.

This proves:
\begin{theorem}\label{AE}
Let $\varphi(x)$ be a finite combination of polynomial equations and
inequations in conjunctive normal form, with $d$ conjuncts
and at most $f$ inequalities per conjunct.
Then there is a polynomial $q(a,b,x)\in\mathbb{C}[a,b,x]$,
whose degree in $a$ is $2d-1$,
whose degree in $b$ is at most $f+1$,
and such that for all $x\in\mathbb{C}^n$,
$$\varphi(x)\Longleftrightarrow(\forall a\in\C)(\exists b\in\C)\,q(a,b,x)=0.$$
\end{theorem}

\section{Optimality: the quantifier prefixes cannot be shortened over $\C$}

The constructions above used exactly two quantifier blocks.
This section shows that such alternation is necessary:
no purely universal or purely existential single-equation
representation is possible over $\C$. The argument for the existential
form is motivated by dimension theory, but we present it in an
elementary form that highlights the connection with quasiminimality.

\begin{lemma}
For any polynomial $p(a,y)$, the set $S=\{y:(\forall a)\,p(a,y)=0\}$ is closed.
\end{lemma}
\begin{corollary}
$y\neq 0$ is not equivalent to any formula of the form
$(\forall a)\,p(a,y)=0$.
\end{corollary}

\begin{theorem}
For any polynomial $p(a,y,z)$, the set
$$S=\{(y,z): \exists a,\,p(a,y,z)=0\}$$
is either empty or uncountable.
\end{theorem}

\begin{proof}
Suppose $S$ is countable.
We will use the quantifier elimination of $\exists$ over $\C$.

Write $$p(a,y,z)=\sum_i p_i(y,z) a^i.$$
For any $(y_0,z_0)\notin S$, there is no root for $p(a,y_0,z_0)$,
so for each $i\ge 1$, $p_i$ vanishes outside of $S$. Since $S$
is countable, its complement is dense, and
by continuity, $p_i$ vanishes on $S$ also. Therefore
$p(a,y,z)=p_0(y,z)$.
Let $q=p_0$, and let $S_y=\{y:\exists z,\,(y,z)\in S\}$.

Write $$q(y,z)=\sum_i q_i(y) z^i.$$
For any $y_0\notin S_y$, there is no root for $q(y_0,z)$,
so for each $i\ge 1$, $q_i$ vanishes outside $S_y$. Since
$S_y$ is countable, its complement is dense, and
by continuity, $q_i$ vanishes on $S_y$ also. Therefore
$q(y,z)=q_0(y)$.

Therefore $S=\{(y,z):q(y,z)=0\}=\{(y,z):q_0(y)=0\}=S_y \times \C$.
This can only be countable if it is empty.
\end{proof}
 
\begin{corollary}
$y=0 \wedge z=0$ is not equivalent to any formula of the form
$(\exists a)\,p(a,y,z)=0$, since $\{(y,z):y=0 \wedge z=0\}$
is neither empty nor uncountable.
\end{corollary}

The impossibility results above are stated for quantification over a single
scalar variable, but they extend to blocks of universally or
existentially quantified variables. Even in higher dimensions,
$S=\{y: (\forall\ab\in\C^m)p(\ab,y)=0\}$ is closed. Similarly, if
$$S=\{(y,z): (\exists\ab\in\C^m)p(\ab,y,z)=0\}$$
then the same argument as above shows that $S$ is either empty or
uncountable, by carrying out the argument in each variable one at a time.
Consequently, regardless of arity, no normal form with
purely existential or purely universal quantification of a single equation
can represent all Boolean combinations of polynomial equations over $\C$.

\section{Normalizing to $\E$ over formally real fields}

A simpler existential form, with only one new quantifier, is available for
any other formally real field, so long as we retain the restriction to formulas
with equations and inequations (but not order). This $\exists$ normal form
holds for $\mathbb{R}$, $\mathbb{Q}$, and $\mathbb{R}^\text{alg}$.

This contrasts with the greater generality of theorems 4.1, 5.1, and 6.1,
which work not only over $\mathbb{C}$ but also over any infinite field of
characteristic $0$. The key properties used there were the invertibility of
elements, the infinitude of the field, and the finitude of roots for any
nonzero polynomial. Here, by contrast, we use the defining property of a
formally real field, namely that a sum of squares is zero iff all the squares
are zero.

We switch notation for quantified variables from $a,b$ to $r,s$ to reflect
the change from $\C$ to a field like $\R$.

\begin{theorem}\label{EforReq}
Let $\varphi(x)$ be a finite combination of polynomial equations and
inequations in disjunctive normal form, with $d$ disjuncts.
Then there is a polynomial $p(r,x)\in\R[r,x]$,
of degree $2d$ in $r$,
and such that for all $x\in\mathbb{R}^n$,
$$\varphi(x)\Longleftrightarrow(\exists r\in\R)\,p(r,x)=0.$$
\end{theorem}
\begin{proof}
We start with $\varphi$ in disjunctive normal form:
\begin{align*}
\varphi
&:=\bigvee_i\Big(\bigwedge_j t_{ij}= 0\;\wedge\;\bigwedge_k u_{ik}\neq 0\Big)
\end{align*}
This $\varphi$ is equivalent to the $\exists$ form below:
\begin{align*}
\varphi\iff\exists r
\prod_i\Big[\sum_j t_{ij}^2 + (1-r\prod_k u_{ik})^2\Big]=0
\end{align*}

For the forward direction, suppose the disjunction holds, and specifically
that the $i^{th}$ disjunct holds. Then let $r=\prod_k u_{ik}^{-1}$, and
$\sum_j t_{ij}^2 + (1-r\prod_k u_{ik})^2=0$,
so the existential statement holds.

For the backwards direction, suppose the existential statement holds with
some $r$. Then for some $i$, $\sum_j t_{ij}^2 + (1-r\prod_k u_{ik})^2=0$,
which means that each of the $t_{ij}$ vanishes and also $r\prod_k u_{ik}=1$,
so each of the $u_{ik}$ is non-zero. Therefore the $i^{th}$ disjunct holds
and the disjunction holds.
\end{proof}

Again with our initial example, now interpreted with $y,z\in\R$,
$$(y=0 \wedge z\neq 0)\vee(z=0\wedge y\neq 0)
\iff
\exists r\, \Big[y^2+(1-rz)^2\Big]\Big[z^2+(1-ry)^2\Big]=0.$$

\section{Normalizing to $\Ed$ and $\AEf$ over $\R$}

For inequalities in $\R$, it is typical to also consider order conditions
such as $u>0$. Indeed Boolean combinations polynomial equations and
such inequalities over $\R$ can be converted to both $\exists^d$ and
$\forall\exists$ forms. We state theorems here analogous to the main
theorem over $\C$, using the algebraic encoding of order from Lemma 3.1.
We omit detailed proofs, since the constructions and proofs are otherwise
parallel to the results over $\C$, and we will later cite other results
establishing a single-quantifier $\exists$ form for these formulas.

\begin{theorem}\label{EdforR}
Let $\varphi(x)$ be a finite combination of polynomial equalities and
order-inequalities in conjunctive normal form, with $d$ conjuncts.
Then there is a polynomial $p(r_1,\dots,r_d,x)\in\mathbb{R}[r_1,\dots,r_d,x]$,
of degree 4 in each $r_i$,
and such that, for all $x\in\mathbb{R}^n$,
$$\varphi(x)\Longleftrightarrow(\exists r_1,\dots,r_d\in\R)\,p(r_1,\dots,r_d,x)=0.$$
\end{theorem}
\emph{Proof Sketch}:
\begin{align*}
\varphi\ :=\  
&\bigwedge_i (\bigvee_j t_{ij}=0 \vee \bigvee_k u_{ik}>0) \\
\iff &\bigwedge_i \exists r\,\prod_j t_{ij}\prod_k(1-r^2 u_{ik})=0 
\phantom{(r-h)\prod_j t_{ij}\prod_k(1-s^2u_{ik})\Big]=0.}
\\
\iff &\exists r_1,\dots,r_d\;\sum_i\Big[\prod_j t_{ij}\prod_k(1-r_i^2u_{ik})\Big]^2=0.
\end{align*}
\pagebreak
\begin{theorem}\label{AEforR}
Let $\varphi(x)$ be a finite combination of polynomial equalities and
order-inequalities in conjunctive normal form, with $d$ conjuncts
and at most $f$ inequalities per conjunct.
Then there is a polynomial $q(r,s,x)\in\mathbb{R}[r,s,x]$,
whose degree in $r$ is $2d-1$,
whose degree in $s$ is at most $2f+1$,
and such that for all $x\in\mathbb{R}^n$,
$$\varphi(x)\Longleftrightarrow(\forall r\in\R)(\exists s\in\R)\,q(r,s,x)=0.$$
\end{theorem}
\emph{Proof Sketch}:
\begin{align*}
\varphi\ :=\
&\bigwedge_i (\bigvee_j t_{ij}=0 \vee \bigvee_k u_{ik}>0) \\
\iff &\bigwedge_i \exists s\,\prod_j t_{ij}\prod_k(1-s^2 u_{ik})=0 \\
\iff &\forall r\,\exists s\,
\Big[1-s\prod_i(r-i)\Big]
\Big[\sum_i\prod_{h\neq i}(r-h)\prod_j t_{ij}\prod_k(1-s^2u_{ik})\Big]=0.
\end{align*}

As an ordered variant of our initial example, if
$$\varphi := (y=0 \vee z=0) \wedge (y>0 \vee z>0)$$
then these constructions yield
$$\varphi \iff (\exists r)(\exists s)\Big[yz\Big]^2 +
\Big[(1-s^2y)(1-r^2z)\Big]^2=0$$
and
$$\varphi \iff (\forall r)(\exists s)\Big[1-s(r-1)(r-2)\Big]
\Big[(r-2)yz+(r-1)(1-s^2y)(1-s^2z)\Big]=0.$$

Since there was no way to express inequations such as $u\neq 0$
in a $\forall$ form over $\R$, there is also no way to express inequalities
in a $\forall$ form over $\R$.

\section{Single-quantifier $\E$ forms over $\R$}

The previous sections gave explicit $\exists^d$ and $\forall\exists$ forms
for Boolean combinations of inequalities over $\R$. David Marker
called our attention to a result of Motzkin showing that in fact a single
existential quantifier always suffices. This result was later sharpened by
Pecker, who achieved the single quantifier via polynomials with degrees
larger than the polynomials above, but much smaller than Motzkin's.
We may write his result in our terms as an $\exists$ normal form:
\begin{theorem}\label{EforR}
Let $\varphi(x)$ be a finite combination of polynomial equalities and
$f$ strict inequalities in conjunctive normal form.
Then there is a polynomial $p(r,x)\in\mathbb{R}[r,x]$,
of degree $2^f$ in $r$,
such that, for all $x\in\mathbb{R}^n$,
$$\varphi(x)\Longleftrightarrow(\exists r\in\R)\,p(r,x)=0.\quad$$
\end{theorem}

As an example, consider the polynomial produced in Pecker's first corollary:
$$(y>0 \wedge z>0) \iff (\exists r)(r^2yz-y^2z-1)^2-y^2z=0$$
We may simplify this to 
$$(y>0 \wedge z>0) \iff (\exists r)(r^2yz-z-1)^2=z$$
Solving the simplified equation for $y$ gives
$y=(z+1\pm\sqrt{z})/r^2z$,
which makes it easier to see that it has
roots exactly when $y$ and $z$ are positive.

As $r$ takes all possible values, graphs like the below sweep
through the entire first quadrant.
\pagebreak
\begin{figure}[h]
    \centering
    \includegraphics[width=0.4\textwidth]{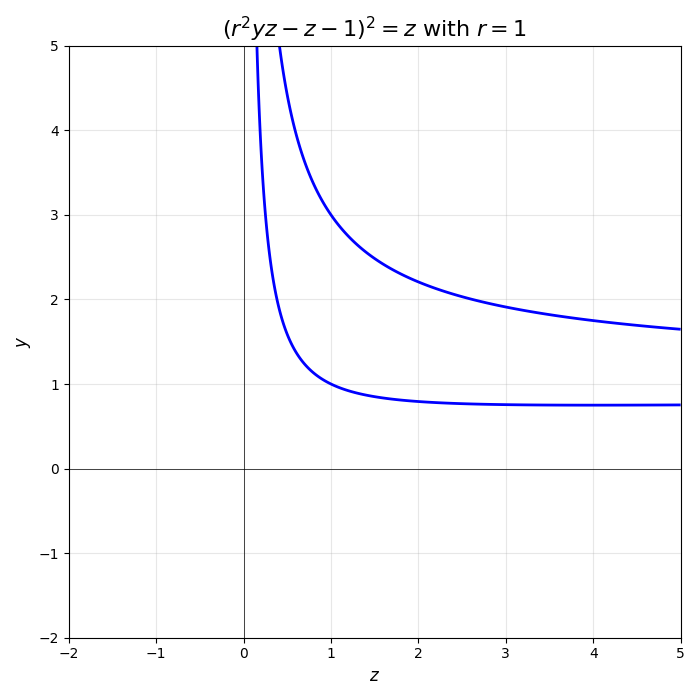}
    \caption{Graph for a simple version of Pecker's polynomial}
    \label{fig:graph}
\end{figure}

A contrast between Pecker's results and ours is that our results, by using
two new variables, require only one use of each term $t_{ij}$ or $u_{ik}$,
and have linearly growing degrees. Pecker's results, using only one new
variable, require multiple uses of those terms, and have exponentially
growing degrees.

\section{Normal forms over $\Q$}
 
The results over $\R$ hold with some modifications over $\Q$, since not
every positive rational is a square. Analogous results obtain from replacing
the squares in $\R$ with sums of four squares in $\Q$. Thus Pecker's form
yields an $\exists^4$ normal form for $\Q$, and the forms provided explicitly
above yield $\exists^{4d}$ and $\forall\exists^4$ forms over $\Q$. 

We can also reduce the $4$'s to $3$'s in these signatures, by using a fact
from number theory and a corollary reformulation. We use $v,w$ for
quantified variables to reflect the switch from $\C$ and $\R$ to $\Q$:
\begin{lemma}
For any natural number $n$, either $n$ or $2n$ is the sum of three squares.
\end{lemma}
\begin{proof}
If $n$ is not of the form $4^k(8m+7)$, then by Legendre's theorem $n$ is a
sum of three squares. If $n$ is $4^k(8m+7)$, then $2n$ is not of
that form, so $2n$ is a sum of three squares.
\end{proof}
\begin{corollary}
$$u\in\Q\wedge u>0 \iff (\exists v_1, v_2, v_3\in\Q)\,
[1-u(v_1^2+v_2^2+v_3^2)]
[1-2u(v_1^2+v_2^2+v_3^2)]=0\Big.$$
\end{corollary}
\begin{proof}
The left-to-right direction follows from the lemma with
$n$ as the product of the numerator and denominator of $1/(2u)$.
The right-to-left direction holds because any sum of three squares is
non-negative, and twice such a sum is non-negative also.
\end{proof} 

In particular, this leads to the following $\exists^{3d}$ and
$\forall\exists^3$ normal forms over $\Q$:
For $\varphi$ in conjunctive normal form as above, we have
$$
\varphi\iff \exists v_1,\dots,v_{3d}\;\sum_i\Big[\prod_j t_{ij}
\prod_k(1-u_{ik}V_i)(1-2u_{ik}V_i)\Big]^2=0.
$$
where $V_i=v_{3i-2}^2+v_{3i-1}^2+v_{3i}^2$. Also
\begin{align*}
\varphi \iff \forall v\,\exists w_1,w_2,w_3\,
&\Big[1-w_1\prod_i(v-i)\Big]
\Big[\sum_i\prod_{h\neq i}(v-h)\prod_j t_{ij}
\prod_k(1-u_{ik}W)(1-2u_{ik}W)\Big]=0.
\end{align*}
where $W=w_1^2+w_2^2+w_3^2$.

These forms may or may not be minimal; we provide them to
show the possibility of a number-theoretic Boolean elimination over $\Q$,
where quantifier elimination is not available.

\section{Conclusion}

The results here show that the Boolean structure of constructible or
semialgebraic sets can be uniformly represented with a single polynomial
equation under either an $\exists\forall$ or a $\forall\exists$ prefix for
$\C$, or even under a simple $\exists$ prefix for $\R$, and that those
quantifier patterns are optimal. We hope it will be useful to have those
results conveniently assembled here.

These results also lead to questions about various extensions of them.
\begin{itemize}
\item Can the $\exists$ results about $\R$ be simplified or optimized for
smaller degrees, and can the $\exists^3$ results about $\Q$ be reduced to
fewer quantifiers?
\item Can every universally quantified formula
be normalized to $\forall\exists$ form with a single equation, and can every
existentially quantified formula be normalized to $\exists\forall$ form?
\item What variants of these constructions are available over $\C$ when the
Boolean combinations may be countable as in $\mathcal{L}_{\omega_1,\omega}$
rather than finite, and the language may be expanded beyond polynomials
to include exponentials or other univariate entire functions?
\end{itemize}
These questions point beyond elementary methods toward a deeper study of
Boolean elimination.

\end{document}